\documentclass[reqno,12pt]{amsart}
\usepackage{amsmath, amssymb}
\usepackage{amsthm, amsbsy}
\usepackage{graphicx}
\usepackage{stackrel}

\theoremstyle{plain}
\newtheorem{theorem}{\indent\bf Theorem}[section]

\newtheorem{corollary}[theorem]{\indent\bf Corollary}

\theoremstyle{definition}
\newtheorem{definition}[theorem]{\indent\bf Definition}
\newtheorem{remark}[theorem]{\indent\bf Remark}

\begin{document}
\title[Irreducibility criteria]{Apollonius circles and the number of irreducible factors of
polynomials}    
\author[A.I. Bonciocat]{Anca Iuliana Bonciocat}
\address{Simion Stoilow Institute of Mathematics of the Romanian 
Academy, Research Unit nr. 3,
P.O. Box 1-764, Bucharest 014700, Romania}
\email{Anca.Bonciocat@imar.ro} 
\author[N.C. Bonciocat]{Nicolae Ciprian Bonciocat}
\address{Simion Stoilow Institute of Mathematics of the Romanian 
Academy, Research Unit nr. 7,
P.O. Box 1-764, Bucharest 014700, Romania}
\email{Nicolae.Bonciocat@imar.ro}
\author[Y. Bugeaud]{Yann Bugeaud}
\address{Universit\'{e} de Strasbourg, Math\'{e}matiques, 7, 
rue Ren\'{e} Descartes, 67084 Strasbourg Cedex, France}
\email{yann.bugeaud@math.unistra.fr}
\author[M. Cipu]{Mihai Cipu}
\address{Simion Stoilow Institute of Mathematics of the Romanian 
Academy, Research Unit nr. 7,
P.O. Box 1-764, Bucharest 014700, Romania}
\email{Mihai.Cipu@imar.ro}
\author[M. Mignotte]{Maurice Mignotte}
\address{Universit\'{e} de Strasbourg, Math\'{e}matiques, 7, 
rue Ren\'{e} Descartes, 67084 Strasbourg Cedex, France}
\email{maurice.mignotte@math.unistra.fr}

\dedicatory{Dedicated to Professors Toma Albu and Constantin N\u ast\u asescu on the occasion of their 80th birthdays}


\begin{abstract}
We provide upper bounds for the sum of the multiplicities of the non-constant irreducible factors that appear in the canonical decomposition of a polynomial $f(X)\in\mathbb{Z}[X]$, in case all the roots of $f$ lie inside an Apollonius circle associated to two points on the real axis with integer abscissae $a$ and $b$, with ratio of the distances to these points depending on the admissible divisors of $f(a)$ and $f(b)$.
In particular, we obtain such upper bounds for the case where $f(a)$ and $f(b)$ have 
few prime factors, and $f$ is an Enestr\"om-Kakeya polynomial, or a Littlewood polynomial, or has 
a large leading coefficient. Similar results are also obtained for multivariate polynomials over arbitrary fields, in a non-Archimedean setting.
\end{abstract}
\maketitle

\section{Introduction} \label{se1}

The prime factorization of the values that an integer polynomial $f(X)$ takes at some specified integral arguments gives useful information on the canonical decomposition of $f$. Many of the classical or more recent irreducibility criteria make use of such information, combined with information on the location of the roots of $f$. One may find such classical results in the works of St\"ackel \cite{Stackel}, Weisner \cite{Weisner}, Ore \cite{Ore} and Dorwart \cite{Dorwart}. For more recent results and some elegant connections between prime numbers and irreducible polynomials we refer the reader to Ram Murty \cite{RamMurty}, Girstmair \cite{Girstmair}, Guersenzvaig \cite{Guersenzvaig1}, and Bodin, D\`ebes and Najib \cite{BDN3}, for instance. 
Some particularly elegant irreducibility criteria write prime numbers or some classes of composite numbers in the number system with base $B$, say, and then replace the base by an indeterminate to produce irreducible polynomials. Here we mention Cohn's irreducibility criterion \cite{PolyaSzego} that uses prime numbers written in the decimal system, and its generalization by Brillhart, Filaseta and Odlyzko \cite{Brillhart} to an
arbitrary base, as well as further generalizations by Filaseta \cite{Filaseta1}, \cite{Filaseta2}, and by Cole, Dunn and Filaseta \cite{CDF}. Another way to produce irreducible polynomials $f$ is to write prime numbers or prime powers as sums of integers of arbitrary sign, one of these integers having a sufficiently large absolute value, and to use these integers as coefficients of $f$ \cite{Bonciocat1}, \cite{Bonciocat2}.
 
In \cite{Apollonius} the irreducibility of an integer polynomial $f$ was studied by combining information on the {\it admissible divisors} of $f(a)$ and $f(b)$ for two integers $a$ and $b$, with information on the location of the roots of $f$. We recall here the definition of admissible divisors, that will be also required throughout this paper.

\begin{definition}\label{admissible}
Let $f$ be a non-constant polynomial with integer coefficients, and let $a$ be an integer with $f(a)\neq 0$. 
We 
say that an integer $d$ is an {\it admissible divisor} of $f(a)$ if $d\mid f(a)$ and 
\begin{equation}\label{adm2}
\gcd \left( d,\frac{f(a)}{d}\right) \mid \gcd (f(a),f'(a)),
\end{equation}
and we shall denote by $\mathcal{D}_{ad}(f(a))$ the set of all admissible divisors of $f(a)$. 
We say that an integer $d$ is a {\it unitary divisor} of $f(a)$ if $d\mid f(a)$ and $d$ is coprime with $f(a)/d$. 
We denote by $\mathcal{D}_{u}(f(a))$ the set of unitary divisors of $f(a)$.  
\end{definition}
We note that this definition was motivated by the fact that if a polynomial $f(X)\in \mathbb{Z}[X]$ factors as 
$f(X)=g(X)h(X)$ with $g,h$ non-constant polynomials in $\mathbb{Z}[X]$, then given an integer $a$ with $f(a)\neq 0$, the integers $g(a)$ and $h(a)$ are divisors of $f(a)$ which must also satisfy the equality $f'(a)=g'(a)h(a)+g(a)h'(a)$. This implies that the greatest common divisor of $g(a)$ and $\frac{f(a)}{g(a)}$ must divide both $f(a)$ and $f'(a)$.
We also note that if $f(a)$ and $f'(a)$ are coprime, then $\mathcal{D}_{ad}(f(a))$ 
reduces to the set $\mathcal{D}_{u}(f(a))$.

As seen in \cite{Apollonius}, one can connect the study of the irreducibility of $f$ with the location of the roots of $f$ inside an Apollonius circle associated 
to the points on the real axis with integer abscissae $a$ and $b$, 
and ratio of the distances to these two points expressed only in terms of some admissible divisors of $f(a)$ and $f(b)$. We recall here the famous definition of a circle given by Apollonius, as the set of points $P$ in the plane that have a given ratio $r$ of distances to two fixed points $A$ and $B$ (see  Figure 1), which may degenerate to a point (for $r\to 0$ or $r\to \infty$) or to a line (for $r\to 1$). 
\begin{center}
\hspace{2.6cm}
\setlength{\unitlength}{8mm}
\begin{picture}(12,5.5)
\linethickness{0.15mm}

\put(-1,2){\vector(1,0){9}}
\put(0,0.25){\vector(0,1){4.75}}

\thicklines
\linethickness{0.1mm}
\put(4,1.3){\line(0,1){0.1}}
\put(4,1.5){\line(0,1){0.1}}
\put(4,1.7){\line(0,1){0.1}}
\put(4,1.9){\line(0,1){0.1}}
\put(4,2.1){\line(0,1){0.1}}
\put(4,2.3){\line(0,1){0.1}}
\put(4,2.5){\line(0,1){0.1}}
\put(4,2.7){\line(0,1){0.1}}
\put(4,2.9){\line(0,1){0.1}}
\put(4,3.1){\line(0,1){0.1}}
\put(4,3.3){\line(0,1){0.1}}
\put(4,3.5){\line(0,1){0.1}}
\put(4,3.7){\line(0,1){0.1}}

{\tiny 
\put(3.25,2.75){$P$}
\put(1.65,3.15){$r>1$}
\put(5.55,3.15){$r<1$}
\put(3.55,4.15){$r=1$}
\put(7.25,3.75){$d(P,B)=r\cdot d(P,A)$}

\put(2.35,1.55){$(a,0)$}
\put(4.75,1.55){$(b,0)$}

\put(2.7,2.2){$A$}
\put(4.97,2.2){$B$}

\put(3.35,0.8){$x=\frac{a+b}{2}$}

}

\put(4,2){\circle{0.08}}

\put(3,2){\circle{0.08}}
\put(3,2){\circle{0.12}}

\put(4.95,2){\circle{0.08}}
\put(4.95,2){\circle{0.12}}

\put(2.666,2){\circle{1.82}}
\put(5.333,2){\circle{1.82}}

\put(3.36,2.53){\circle{0.08}}
\put(3.36,2.53){\circle{0.12}}

\put(3,2){\line(2,3){0.35}}  

\put(3.36,2.53){\line(3,-1){1.6}}  

{\small 
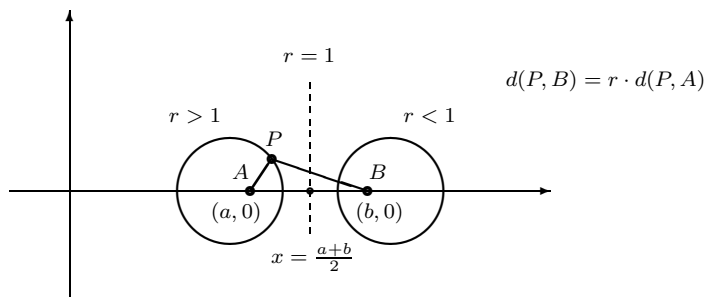
\put(-5.6,-0.65){{\bf Figure 1.}\ The Apollonius circles associated to a pair of points in the plane}
}
\end{picture}
\end{center}
\bigskip

\medskip

More precisely, given two points $A=(a,0)$ and $B=(b,0)$ 
and $r>0$,  $r\neq 1$,          
the set of points $P=(x,y)$ with $d(P,B)=r\cdot d(P,A)$ is the 
Apollonius circle ${\rm Ap}(a,b,r)$ given by the equation
\begin{equation}\label{Apollonius1}
\left( x-\frac{ar^2-b}{r^2-1}\right)^2+y^2=\left(r\frac{b-a}{r^2-1}\right) ^2,
\end{equation}
which degenerates to the point $(b,0)$ for $r\to 0$, to the point $(a,0)$ for $r\to \infty$, and to the vertical line $x=\frac{a+b}{2}$ for $r\to 1$.

Consider the canonical decomposition $f(X)=const\cdot \prod\limits_{i=1}^{r}f
_{i}(X)^{m_{i}}$ of a polynomial $f\in\mathbb{Z}[X]$, with the $f_{i}$'s irreducible and prime to each other. The aim of this paper is to generalize the results in \cite{Apollonius} by obtaining upper bounds for $m_1+\cdots +m_r$. This will be achieved by adapting the methods in \cite{Apollonius} and considering a sequence of potentially larger Apollonius circles that might contain the roots of $f$, and here too, the admissible divisors of $f(a)$ and $f(b)$ will play a crucial role. For other results that bound the sum of these multiplicities we refer the reader to Guersenzvaig \cite{Guersenzvaig1}, and to \cite{BoncioActaArit}, where some methods of Cavachi, M. V\^aj\^aitu and Zaharescu \cite{CAV}, \cite{CVZ1} to study linear combinations of relatively prime polynomials have been employed.

Throughout the paper, instead of saying that the sum of the multiplicities of the non-constant irreducible factors that appear in the canonical decomposition of $f$ is at most $k$, we will simply say that $f$ is the product of at most $k$ irreducible factors over $\mathbb{Q}$. 

Our first result that establishes a connection between Apollonius circles and the sum of these multiplicities is the following.
\begin{theorem}\label{thm0}
Let $f(X)=a_{0}+a_{1}X+\cdots +a_{n}X^{n}\in \mathbb{Z}[X]$, and assume that for two integers $a,b$ we have $0<|f(a)|<|f(b)|$. Let $k$ be a positive integer, and let
\begin{equation}\label{primulq}
q_k=\max \left\{ \frac{d_2}{d_1}\leq \sqrt[k+1]{\frac{|f(b)|}{|f(a)|}}: 
d_1\in \mathcal{D}_{ad}(f(a)), \ d_2\in \mathcal{D}_{ad}(f(b)) \right\} .
\end{equation}
Then $f$ is the product of at most $k$ irreducible factors over $\mathbb{Q}$ in each one of the following three cases:

\emph{i)} \ \thinspace $q_k>1$ and all the roots of $f$ lie inside the Apollonius  circle ${\rm Ap}(a,b,q_k)$;

\emph{ii)} \thinspace $q_k>1$, all the roots of $f$ lie inside the Apollonius circle ${\rm Ap}(a,b,\sqrt{q_k})$, and $f$ has no rational roots;

\emph{iii)} $q_k=1$ and either $b>a$ and all the roots of $f$ lie in 
the half-plane $x<\frac{a+b}{2}$, or $a>b$ and all the roots of $f$ lie in the half-plane $x>\frac{a+b}{2}$.
\end{theorem}

We mention that by taking $k=1$ in Theorem \ref{thm0}, we recover Theorem 1 in \cite{Apollonius}. Theorem \ref{thm0} can be given in a more explicit form by taking into account the equation of the Apollonius circles given by (\ref{Apollonius1}), as follows.
\begin{theorem}\label{thm0Explicit}
Let $f(X)=a_{0}+a_{1}X+\cdots +a_{n}X^{n}\in \mathbb{Z}[X]$, and assume that for two integers $a,b$ we have $0<|f(a)|<|f(b)|$. Let $k$ be a positive integer and let $q_k$ be given by {\em (\ref{primulq})}. Then $f$ is the product of at most $k$ irreducible factors over $\mathbb{Q}$ in each one of the following three cases:

\emph{i)} \ \thinspace $q_k>1$ and each root $\theta $ of $f$ satisfies $|\theta-\frac{aq_k^2-b}{q_k^2-1}|<q_k\frac{|b-a|}{q_k^2-1}$;

\emph{ii)} \thinspace $q_k>1$, each root $\theta $ of $f$ satisfies $|\theta -\frac{aq_k-b}{q_k-1}|<\sqrt{q_k}\frac{|b-a|}{q_k-1}$, and $f$ has no rational roots;

\emph{iii)} $q_k=1$ and either $b>a$ and all the roots of $f$ lie in 
the half-plane $x<\frac{a+b}{2}$, or $a>b$ and all the roots of $f$ lie in the half-plane $x>\frac{a+b}{2}$.
\end{theorem}
Notice that if $b>a$ and we can prove that $q_k=1$ for some $k\geq 1$, by imposing the condition that $f(X+\frac{a+b}{2})$ is a Hurwitz stable polynomial, so that all the roots of $f$ lie in the half-plane $x<\frac{a+b}{2}$, then by Theorem \ref{thm0} iii) we may conclude that $f$ is the product of at most $k$ irreducible factors over $\mathbb{Q}$.  Recall that a necessary and sufficient condition for a polynomial to be Hurwitz stable is that it passes the Routh--Hurwitz test.

In some situations, instead of testing the conditions in Theorem \ref{thm0} or Theorem \ref{thm0Explicit}, 
it might be more convenient to use the maximum of the absolute values 
of the roots of $f$:

\begin{theorem}\label{thm1} 
Let $f(X)=a_{0}+a_{1}X+\cdots +a_{n}X^{n}\in \mathbb{Z}[X]$, let $M$ be the maximum of the absolute values of its roots,
and assume that for two integers $a,b$ we have $0<|f(a)|<|f(b)|$. Let $k$ be a positive integer and let $q_k$ be given by {\em (\ref{primulq})}. Then $f$ is the product of at most $k$ irreducible factors over $\mathbb{Q}$ in each one of the following three cases:

\emph{i)} \ \thinspace \thinspace $|b|>q_k|a|+(1+q_k)M$; 

\emph{ii)} \thinspace \thinspace $|b|>\sqrt{q_k}|a|+(1+\sqrt{q_k})M$ and 
$f$ has no rational roots;

\emph{iii)} \thinspace $q_k=1$, $a^2<b^2$ and $M<\frac{|a+b|}{2}$.
\end{theorem}

As we shall see later in Remark \ref{Remark1}, in some cases, conditions i) and ii) in Theorem \ref{thm0Explicit} may lead to sharper conditions than the corresponding ones in Theorem \ref{thm1}, even if we don't explicitly compute the coefficients of $f(X+\frac{aq_k^2-b}{q_k^2-1})$ and $f(X+\frac{aq_k-b}{q_k-1})$ to derive then estimates for the maximum of the absolute values of their roots. However, by computing these coefficients and avoiding unnecessary use of the triangle inequality, one might obtain even sharper conditions on $a$ and $b$.

When there is no information available on the prime factorization of $f'(a)$ and $f'(b)$, thus preventing one to use the admissible divisors of $f(a)$ and $f(b)$, we may content ourselves with slightly weaker results by allowing $d_1$ 
and $d_2$ in the definition of $q_k$
to be arbitrary divisors of $f(a)$ and $f(b)$, respectively. This will potentially increase $q_k$, leading to stronger restrictions on $a$ and $b$. The computation of $q_k$ requires analyzing inequalities between products of prime powers, so an explicit, effective formula of $q_k$ can be obtained only in a few cases where $f(a)$ and $f(b)$ have a small number of prime factors. 
On the other hand, one may relax the restrictions on $a$ and $b$ by finding sharp estimates for $M$, the maximum of the absolute values of the roots of $f$. The reader may benefit of the extensive literature on this subject, originating in the works of Cauchy and Lagrange. Here we will only refer the reader to the generalization for Cauchy's bound 
on the largest root of a polynomial \cite{MM3}, to a recent improvement of the bound of Lagrange \cite{BMS}, to some classical results relying on families of parameters obtained by Fujiwara \cite{Fujiwara}, Ballieu \cite{Ballieu}, 
\cite{Marden}, Cowling and Thron \cite{Cowling1}, \cite{Cowling2}, Kojima \cite{Kojima}, and to methods that use estimates for the characteristic 
roots for complex matrices \cite{Perron}. 
 
The results stated so far will be proved in Section \ref{se2}. We will present in Section \ref{applic} some corollaries of Theorem \ref{thm1}, for some cases where the prime factorizations of $f(a)$ and $f(b)$ allow one to conclude that $q_k=1$. Results that rely on information on the unitary divisors of $f(a)$ and $f(b)$ will be given in Section \ref{se3}. In Section \ref{se4} we will provide some analogous results for multivariate polynomials over arbitrary fields. Some examples will be given in the last section of the paper.

\section{Proofs for the case of admissible divisors}  \label{se2}

{\it Proof of Theorem \ref{thm0}} \ First of all let us notice that, if we fix the integers $a$ and $b$ as in the statement of the theorem, then $q_k$ is a decreasing function on $k$. Moreover, since $1$ is obviously an admissible divisor 
of $f(a)$ and $f(b)$, a possible candidate for $q_{k}$ is $1$, so $q_{k}\geq 1$ for each $k$.  Assume now to the contrary that $f$ is the product of $s$ factors $f_1,\dots ,f_s\in \mathbb{Z}[X]$ with $s\geq k+1$ and $f_1,\dots ,f_s$ irreducible over $\mathbb{Q}$. Next, let us select an index $i\in\{ 1,\dots ,s\} $ and let us denote
\[
g_i(X):=\prod\limits_ {j\neq i}f_j(X).
\]
We may thus write $f=f_i\cdot g_i$ for each $i=1,\dots ,s$. Now, since $f(a)=f_i(a)g_i(a)\neq 0$ and 
$f'(a)=f_i'(a)g_i(a)+f_i(a)g_i'(a)$, and similarly $f(b)=f_i(b)g_i(b)\neq 0$ and 
$f'(b)=f_i'(b)g_i(b)+f_i(b)g_i'(b)$, we see that $f_i(a)$ is a divisor 
$d_{i}$ of $f(a)$, and $f_i(b)$ is a divisor $d'_i$ of $f(b)$
that must also satisfy the following divisibility conditions
\[
\gcd \left( d_i,\frac{f(a)}{d_i}\right) \mid  \gcd (f(a),f'(a))\ 
\mbox{\rm \ and\ }\ \gcd \left( d'_i,\frac{f(b)}{d'_i}\right) \mid  \gcd (f(b),f'(b)). 
\]
Therefore $d_i$ and $d'_i$ are admissible divisors of $f(a)$ and $f(b)$, 
respectively, and this holds for each $i=1,\dots ,s$. Next, since
\[
\frac{d'_1}{d_1}\cdots \frac{d'_s}{d_s}=\frac{f(b)}{f(a)},
\]
one of the quotients $\frac{|d'_1|}{|d_1|},\dots ,\frac{|d'_s|}{|d_s|}$, 
say $\frac{|d'_1|}{|d_1|}$, must be less than or equal to 
$\sqrt[s]{\frac{|f(b)|}{|f(a)|}}$. In particular, this shows that we must have
\begin{equation}\label{q}
\frac{|f_1(b)|}{|f_1(a)|}\leq q_{s-1}.
\end{equation}
Assume now that $f$ factors as 
$f(X)=a_{n}(X-\theta _{1})\cdots (X-\theta _{n})$ for some complex numbers
$\theta _{1},\dots ,\theta _{n}$. Without loss of generality we may further assume that $\deg f_1=m\geq 1$, say, and that
\[
f_1(X)=b_{m}(X-\theta _{1})\cdots (X-\theta _{m}) 
\]
for some divisor $b_{m}$ of $a_{n}$.
Next, we observe that we may write
\[
\frac{f_1(b)}{f_1(a)}=\frac{b-\theta_1}{a-\theta_1}\cdots \frac{b-\theta_{m}}{a-\theta_{m}},
\]
so in view of (\ref{q}) for at least one index $j\in \{ 1,\dots ,m\} $ 
we must have
\begin{equation}\label{radical}
\frac{|b-\theta_j|}{|a-\theta_j|}\leq q_{s-1}^{\frac{1}{m}}.
\end{equation}
Now, let us first assume that $q_k>1$ and all the roots of $f$ lie inside 
the Apollonius circle ${\rm Ap}(a,b,q_k)$. In particular, since $\theta _j$ lies 
inside the Apollonius circle ${\rm Ap}(a,b,q_k)$, it must satisfy the inequality $|b-\theta_j|>q_k|a-\theta_j|$.
Since $q_k>1$ and $m\geq 1$, we have $q_k\geq q_k^{\frac{1}{m}}$, so we deduce 
that we actually have
\[
\frac{|b-\theta_j|}{|a-\theta_j|}>q_k^{\frac{1}{m}}\geq q_{s-1}^{\frac{1}{m}},
\]
as $q_k\geq q_{s-1}$, and this contradicts (\ref{radical}). Therefore $f$ is the product of at most $k$ irreducible factors over $\mathbb{Q}$.

Next, assume that $q_k>1$ and that all the roots of $f$ lie inside the 
Apollonius circle ${\rm Ap}(a,b,\sqrt{q_k})$. In particular, we have 
$|b-\theta_j|>\sqrt{q_k}|a-\theta_j|$. Since $f$ has no rational roots, 
we must have $m\geq 2$, so $\sqrt{q_k}\geq q_k^{\frac{1}{m}}$, which also 
leads us to the desired contradiction 
\[
\frac{|b-\theta_j|}{|a-\theta_j|}>q_k^{\frac{1}{m}}\geq q_{s-1}^{\frac{1}{m}},
\]
so in this case too one may write $f$ as a product of at most $k$ irreducible factors over $\mathbb{Q}$.

Finally, let us assume that $q_k=1$. Since $1\leq q_{s-1}\leq q_k$, this implies that $q_{s-1}$ is also equal to $1$, so in this case inequality (\ref{radical}) reads
\[
\frac{|b-\theta_j|}{|a-\theta_j|}\leq 1,
\]
which is equivalent to
\begin{equation}\label{inegab}
(b-Re(\theta _j))^2\leq (a-Re(\theta _j))^2.
\end{equation}
To contradict (\ref{inegab}), it is therefore
sufficient to ask all the roots of $f$ to lie in the half-plane $x<\frac{a+b}{2}$ 
if $a<b$, or in the half-plane $x>\frac{a+b}{2}$ if $a>b$, and this completes the proof.
\hfill  $\square $

\medskip

{\it Proof of Theorem \ref{thm0Explicit}} \ i) In view of (\ref{Apollonius1}), the abscissa of the center of the Apllonius circle ${\rm Ap}(a,b,q_k)$ is $C_k=\frac{aq_k^2-b}{q_k^2-1}$, and its radius is $R_k=q_k\frac{|b-a|}{q_k^2-1}$. The condition that all the roots of $f$ lie inside the Apollonius circle ${\rm Ap}(a,b,q_k)$ is equivalent to asking all the roots of the polynomial $f(X+C_k)$ to lie inside the circle centered in the origin and with radius $R_k$, or equivalenly, that all the roots of $f(X+C_k)$ have absolute values less than $R_k$. Thus, if the roots of $f$ are $\theta _1,\dots ,\theta _n$, say, then the roots of $f(X+C_k)$ are $\theta _1-C_k,\dots ,\theta _n-C_k$, and our condition on $\theta _1,\dots ,\theta _n$ reads
\begin{equation}\label{translatie}
\max\limits _{1\leq i\leq n} |\theta _i-C_k|<R_k.
\end{equation}

ii) Here the abscissa of the center of the Apllonius circle ${\rm Ap}(a,b,\sqrt{q_k})$ is $C'_k=\frac{aq_k-b}{q_k-1}$, and its radius is $R'_k=\sqrt{q_k}\frac{|b-a|}{q_k-1}$. One argues as in the previous case with $C'_k$ instead of $C_k$ and $R'_k$ instead of $R_k$.
\hfill  $\square $       

\begin{remark}\label{Remark1}\ We mention that, using the notation in the proof of Theorem \ref{thm0Explicit}, we have
\begin{equation}\label{complicat}
f(X+C_k)=\sum\limits _{i=0}^{n}b_iX^i\quad {\rm with}\quad b_i=\sum\limits _{j=i}^{n}a_j\tbinom jiC_k^{j-i},
\end{equation}
so when it is possible, to test (\ref{translatie}) one should compute the coefficients $b_i$ in (\ref{complicat}) and then apply some known estimates for the roots of $f(X+C_k)$ in terms of the $b_i$'s. When the computation of the $b_i$'s is rather difficult, one may use a condition stronger than (\ref{translatie}), that is
\[
\max\limits _{1\leq i\leq n}|\theta _i|<R_k-|C_k|,
\]
provided $R_k>|C_k|$, or equivalently, that $q_k|b-a|>|aq_k^2-b|$. In this way one may directly use some suitable estimates for $M:=\max\limits _{1\leq i\leq n}|\theta _i|$ in terms of the coefficients of $f$, thus avoiding the computation of the $b_i$'s, and in some cases the outcome may consist of sharper conditions on $a$ and $b$ than those in Theorem \ref{thm1} i). To see this, we will assume that $q_k>1$ and we will present the cases where $R_k>|C_k|$, depending on the signs of $a$ and $b$ and on their magnitude:

\smallskip

1) If $a>0>b$, $|b|>a$ and $q_k<\frac{|b|}{a}$, then $R_k-|C_k|=-\frac{aq_k+b}{q_k+1}>0$, and our condition $M<R_k-|C_k|$ leads us to $|b|>q_k|a|+(1+q_k)M$;

2) If $b>a>0$ and $q_k<\sqrt{\frac{b}{a}}$, then 
$R_k-|C_k|=\frac{aq_k+b}{q_k+1}>0$, and condition $M<R_k-|C_k|$ leads to $|b|>-q_k|a|+(1+q_k)M$;

3) If $b>a>0$ and $\sqrt{\frac{b}{a}}\leq q_k<\frac{b}{a}$, then 
$R_k-|C_k|=\frac{-aq_k+b}{q_k-1}>0$, and condition $M<R_k-|C_k|$ leads to $|b|>q_k|a|+(q_k-1)M$;

4) If $a=0$ and $b\neq 0$, then $R_k-|C_k|=\frac{|b|}{q_k+1}>0$, and condition $M<R_k-|C_k|$ leads to $|b|>(1+q_k)M$;

5) If $a<0<b$, $b>|a|$ and $q_k<\frac{b}{|a|}$, then 
$R_k-|C_k|=\frac{aq_k+b}{q_k+1}>0$, and condition $M<R_k-|C_k|$ leads to $|b|>q_k|a|+(1+q_k)M$;

6) If $b<a<0$ and $q_k<\sqrt{\frac{b}{a}}$, then 
$R_k-|C_k|=-\frac{aq_k+b}{q_k+1}>0$, and condition $M<R_k-|C_k|$ leads to $|b|>-q_k|a|+(1+q_k)M$;

7) If $b<a<0$ and $\sqrt{\frac{b}{a}}\leq q_k<\frac{b}{a}$, then 
$R_k-|C_k|=\frac{aq_k-b}{q_k-1}>0$, and condition $M<R_k-|C_k|$ leads to $|b|>q_k|a|+(q_k-1)M$.
\smallskip

Thus, in cases 2),\thinspace 3) and 6),\thinspace 7) we reach the same conclusion as in Theorem \ref{thm1} i), but with a less restrictive condition than $|b|>q_k|a|+(1+q_k)M$.

Similar considerations apply to case ii) of Theorem \ref{thm1} as well.
\end{remark}

{\it Proof of Theorem \ref{thm1}} \ The proof follows the same lines as in the case of 
Theorem \ref{thm0}, and we deduce again that for at least one index 
$j\in \{ 1,\dots ,m\} $ we must have
\begin{equation}\label{radical2}
\frac{|b-\theta_j|}{|a-\theta_j|}\leq q_{s-1}^{\frac{1}{m}}.
\end{equation}
On the other hand, if $|b|>q_k|a|+(1+q_k)M$ we observe that 
\[
\frac{|b-\theta_j|}{|a-\theta_j|}\geq \frac{|b|-|\theta_j|}{|a|+|\theta_j|}
\geq \frac{|b|-M}{|a|+M}>q_k\geq q_k^{\frac{1}{m}}\geq q_{s-1}^{\frac{1}{m}}, 
\]
since $q_k\geq 1$ and $q_k\geq q_{s-1}$. This contradicts (\ref{radical2}), so $f$ is the product of at most $k$ irreducible factors over $\mathbb{Q}$. 

In our second case, if we assume that $|b|>\sqrt{q_k}|a|+(1+\sqrt{q_k})M$ 
and $f$ has no rational roots, then $m\geq 2$, and consequently 
 \[
\frac{|b-\theta_j|}{|a-\theta_j|}\geq \frac{|b|-M}{|a|+M}>\sqrt{q_k}\geq q_k^{\frac{1}{m}}\geq q_{s-1}^{\frac{1}{m}}, 
\]
again a contradiction.

In our third case, let us assume that $q_k=1$, $a^2<b^2$ and
$M<\frac{|a+b|}{2}$. 
If $b>a$, then $a+b>0$ and the conclusion follows by Theorem \ref{thm0} iii) 
since the disk $\{ z:|z|\leq M\} $ containing all the roots of $f$ lies in the 
left half-plane $x<\frac{a+b}{2}$. On the other hand, if $a>b$, then $a+b<0$ and 
the disk $\{ z:|z|\leq M\} $ lies in the right half-plane $x>\frac{a+b}{2}$, 
as $-M>\frac{a+b}{2}$.  \hfill  $\square $       


\section{Applications}  \label{applic}

Our first application of Theorem \ref{thm1} is the following result.
\begin{corollary}\label{coro1main}
Let $f(X)=a_{0}+a_{1}X+\cdots +a_{n}X^{n}\in \mathbb{Z}[X]$, and $a$, $b$ two integers such that $a^2<b^2$ and
$|a_{n}|>\sum_{i=0}^{n-1}|a_{i}| \bigl(\frac{|a+b|}{2}\bigr) ^{i-n}$. Assume that $|f(a)|=p^{k_1}r$, $|f(b)|=p^{k_2}$ with $p$ a prime number and 
$k_1,k_2,r$ integers with $k_2>k_1\geq 0$ and $0<r<p$. Then $f$ is the product of at most $1+\lfloor (k_2-k_1-1)\log _{\frac{p}{r}}p \rfloor $ irreducible factors over $\mathbb{Q}$.
\end{corollary}
\begin{proof} \ 
An immediate consequence of 
Rouch\'e's Theorem is that the condition 
$|a_{n}|>\sum_{i=0}^{n-1}|a_{i}|\bigl(\frac{|a+b|}{2}\bigr) ^{i-n}$ forces 
all the roots of $f$ to have absolute values less than $\frac{|a+b|}{2}$. Therefore $M<\frac{|a+b|}{2}$. 
Note that for every positive integer $k$ we have $q_k\leq \tilde{q}_k$, with 
\begin{equation}\label{qmaimare}
\tilde{q}_k=\max \left\{ \frac{d_2}{d_1}\leq \sqrt[k+1]{\frac{|f(b)|}{|f(a)|}}: 
d_1\mid f(a), \ d_2\mid f(b) \right\} .
\end{equation}
We search for a positive integer $k$ as small as possible such that $\tilde{q}_k=1$, which will also imply that $q_k=1$. In view of (\ref{qmaimare})
we have
\[
\tilde{q}_k=\max \left\{ \frac{d_2}{d_1}\leq \sqrt[k+1]{\frac{p^{k_2-k_1}}{r}}: 
d_1\mid p^{k_1}r, \ d_2\mid p^{k_2} \right\} .
\]
Observe that the least quotient $\frac{d_2}{d_1}$ that exceeds $1$ with $d_1\mid p^{k_1}r$ and $d_2\mid p^{k_2}$ is $\frac{p}{r}$, so to prove that $\tilde{q}_k=1$ it suffices to prove that $\frac{p}{r}>\sqrt[k+1]{\frac{p^{k_2-k_1}}{r}}$, or equivalently that $k>(k_2-k_1-1)\log _{\frac{p}{r}}p$. Note that for $k_2-k_1>1$ and $r>1$, the term $(k_2-k_1-1)\log _{\frac{p}{r}}p$ cannot be an integer, as $p$ is a prime number. A suitable candidate for $k$ is obviously $1+\lfloor (k_2-k_1-1)\log _{\frac{p}{r}}p \rfloor $, and the conclusion follows from Theorem \ref{thm1}.  \end{proof}
We mention here that for an upper bound which depends only on $p$ and $k_2$, one may use $1+\lfloor (k_2-1)\log _{\frac{p}{p-1}}p \rfloor $, for instance. 
A special instance of Corollary \ref{coro1main} is the following result.

\begin{corollary}
\label{coro2} Let $f(X)=a_{0}+a_{1}X+\cdots +a_{n}X^{n}$ be a polynomial with 
integer coefficients, with $|a_{n}|>2|a_{n-1}|+2^{2}|a_{n-2}|+\cdots +2^{n}|a_{0}|$ 
and $a_{0}\neq 0$. If $f(\mathbb{Z}\setminus \{0\})$ or $-f(\mathbb{Z}\setminus \{0\})$ 
contains a prime power $p^k$ with $p>|a_0|$ and $k\geq 1$, then $f$ is the product of at most
\[
1+\lfloor (k-1)\log _{\frac{p}{|a_0|}}p \rfloor
\]
irreducible factors over $\mathbb{Q}$.
\end{corollary}

\begin{proof} \ We will apply Corollary \ref{coro1main} with $a=0$ and $b$ an integer such that $f(b)=p^k$, so here $k_2=k$. Since $p>|a_0|=|f(0)|$ we have $k_1=0$ and $r=|a_0|<p$. As $b\neq 0$ we have $|b|\geq 1$, so our assumption on the magnitude of $a_n$ implies that
\begin{eqnarray*}
|a_{n}| & > & 2|a_{n-1}|+2^{2}|a_{n-2}|+\cdots +2^{n}|a_{0}|\\
 & \geq & 2|a_{n-1}|+\frac{2^{2}}{|b|^2}|a_{n-2}|+\cdots +\frac{2^{n}}{|b|^n}|a_{0}|=\sum_{i=0}^{n-1}|a_{i}| \left(\frac{|a+b|}{2}\right) ^{i-n}.
\end{eqnarray*}
Thus condition $|a_n|>\sum_{i=0}^{n-1}|a_{i}| \bigl(\frac{|a+b|}{2}\bigr) ^{i-n}$ in Corollary \ref{coro1main} is satisfied, so we may conclude that $f$ is the product of at most
$1+\lfloor (k-1)\log _{\frac{p}{|a_0|}}p \rfloor $
irreducible factors over $\mathbb{Q}$.
\end{proof}

Another application of Theorem \ref{thm1} is the following result.

\begin{corollary}\label{coro1main'}
Let $f(X)=a_{0}+a_{1}X+\cdots +a_{n}X^{n}\in \mathbb{Z}[X]$, and $a$, $b$ two integers such that $a^2<b^2$ and
$|a_{n}|>\sum_{i=0}^{n-1}|a_{i}| \bigl(\frac{|a+b|}{2}\bigr) ^{i-n}$. Assume that $|f(a)|=p^{k_1}$, $|f(b)|=p^{k_2}r$ with $p,r$ prime numbers with $p>r$ and 
$k_1,k_2$ integers with $k_2\geq k_1>0$. Then $f$ is the product of at most $1+\lfloor (k_2-k_1)\log _{r}p \rfloor $ irreducible factors over $\mathbb{Q}$.
\end{corollary}

We will prove here a more general version of Corollary \ref{coro1main'}, in which we allow $r$ to be composite:
\medskip

\begin{corollary}\label{coro1main''}
Let $f(X)=a_{0}+a_{1}X+\cdots +a_{n}X^{n}\in \mathbb{Z}[X]$, and $a$, $b$ two integers such that $a^2<b^2$ and
$|a_{n}|>\sum_{i=0}^{n-1}|a_{i}| \bigl(\frac{|a+b|}{2}\bigr) ^{i-n}$. Assume that $|f(a)|=p^{k_1}$, $|f(b)|=p^{k_2}r$ with $p$ a prime number, $r$ an integer with $0<r<p$ and 
$k_1,k_2$ integers with $k_2\geq k_1>0$. Let also $q$ be the smallest prime factor of $r$. Then $f$ is the product of at most $1+\lfloor (k_2-k_1)\log _{q}p +\log _{q}\frac{r}{q}\rfloor $ irreducible factors over $\mathbb{Q}$.
\end{corollary}

\begin{proof} \ Here too,
all the roots of $f$ have absolute values less than $\frac{|a+b|}{2}$, so $M<\frac{|a+b|}{2}$. In view of (\ref{qmaimare}), we deduce in this case that
\[
\tilde{q}_k=\max \left\{ \frac{d_2}{d_1}\leq \sqrt[k+1]{p^{k_2-k_1}r}: 
d_1\mid p^{k_1}, \ d_2\mid p^{k_2}r \right\} .
\]
Observe that here the least quotient $\frac{d_2}{d_1}$ that exceeds $1$ with $d_1\mid p^{k_1}$ and $d_2\mid p^{k_2}r$ is $q$, so to prove that $\tilde{q}_k=1$ it suffices to prove that $q>\sqrt[k+1]{p^{k_2-k_1}r}$, or equivalently that 
\[
k>(k_2-k_1)\log _{q}p+\log _{q}\frac{r}{q}.
\]
We may therefore choose $k=1+\lfloor (k_2-k_1)\log _{q}p +\log _{q}\frac{r}{q}\rfloor $, and the conclusion follows again from Theorem \ref{thm1}.  \end{proof}

Using the well-known Enestr\"om--Kakeya Theorem \cite{Kakeya}, saying 
that all the roots of a polynomial $f(X)=a_{0}+a_{1}X+\cdots +a_{n}X^{n}$ 
with real coefficients satisfying $0\leq a_{0}\leq a_{1}\leq \dots \leq a_{n}$ 
must have absolute values at most $1$, one can also prove the following results, in which we will only consider the case of positive integers $a,b$.
\begin{corollary}\label{EK}
Let $f\in\mathbb{Z}[X]$ be an Enestr\"om--Kakeya polynomial with $f(0)\neq 0$, and let $a$, $b$ be
two positive integers such that $f(a)=p^{k_1}r$, $f(b)=p^{k_2}$ with $p$ a prime number and 
$k_1,k_2,r$ integers with $k_2>k_1\geq 0$ and $0<r<p$. Then $f$ is the product of at most $1+\lfloor (k_2-k_1-1)\log _{\frac{p}{r}}p \rfloor $ irreducible factors over $\mathbb{Q}$.
\end{corollary} 

\begin{proof} \ Here, by the Enestr\"om--Kakeya 
Theorem all the roots of $f$ must have modulus at most $1$, so $M\leq 1$. Since $f(a)<f(b)$ and $f$ has positive coefficients, we must have $a<b$, which also shows that $a+b>2$, hence condition $M<\frac{|a+b|}{2}$ in Theorem \ref{thm1} iii) is obviously satisfied. To prove that $q_k=1$ with $k=1+\lfloor (k_2-k_1-1)\log _{\frac{p}{r}}p \rfloor $, one may argue as in the proof of Corollary \ref{coro1main}, and the conclusion follows by Theorem \ref{thm1} iii).  \end{proof}

Another result for Enestr\"om--Kakeya polynomials is the following.
\begin{corollary}\label{EK2}
Let $f\in\mathbb{Z}[X]$ be an Enestr\"om--Kakeya polynomial with $f(0)\neq 0$, and let $a$, $b$ be
two positive integers such that $f(a)=p^{k_1}$, $f(b)=p^{k_2}r$ with $p,r$ prime numbers with $p>r$, and 
$k_1,k_2$ integers with $k_2\geq k_1>0$. Then $f$ is the product of at most $1+\lfloor (k_2-k_1)\log _{r}p \rfloor $ irreducible factors over $\mathbb{Q}$.
\end{corollary} 

\begin{proof} \ We argue as before, this time using the  fact that $q_k=1$ with  $k=1+\lfloor (k_2-k_1)\log _{r}p \rfloor $, as in
Corollary \ref{coro1main'}.
\end{proof}
One may also prove the following analogous results for Littlewood polynomials, that is for polynomials all of whose coefficients are $\pm1$. 

\begin{corollary}\label{LW}
Let $f$ be a Littlewood polynomial, and $a$, $b$ 
two nonnegative integers with $a+b\geq 4$ such that $f(a)=p^{k_1}r$, $f(b)=p^{k_2}$ with $p$ a prime number and 
$k_1,k_2,r$ integers with $k_2>k_1\geq 0$ and $0<r<p$. Then $f$ is the product of at most $1+\lfloor (k_2-k_1-1)\log _{\frac{p}{r}}p \rfloor $ irreducible factors over $\mathbb{Q}$.
\end{corollary} 

\begin{proof} \ Here we use the fact that all the roots of a Littlewood polynomial have absolute values less than $2$, so $M<2$. Indeed, assume to the contrary that $f$ would have a root $\theta $ with $|\theta |\geq 2$. Then we would obtain
\begin{eqnarray*}
0 & = & |\sum\limits _{i=0}^{n}a_i\theta ^{i-n}|\geq |a_n|-\left(\sum\limits_{i=0}^{n-1}|a_i|\cdot |\theta |^{i-n}\right)\\
 & = & 1-\frac{1}{|\theta |}-\frac{1}{|\theta |^2}-\cdots -\frac{1}{|\theta |^{n}}\geq 1-\frac{1}{2}-\frac{1}{2^2}-\cdots -\frac{1}{2^n}>0,
\end{eqnarray*}
a contradiction. Since $a+b\geq 4$ and $M<2$, condition $M<\frac{|a+b|}{2}$ in Theorem \ref{thm1} iii) is obviously satisfied. Again, one proves as in  Corollary \ref{coro1main} that we have $q_k=1$ with $k=1+\lfloor (k_2-k_1-1)\log _{\frac{p}{r}}p \rfloor $, and the conclusion follows by Theorem \ref{thm1} iii).  \end{proof}
\begin{corollary}\label{LW2}
Let $f$ be a Littlewood polynomial, and $a$, $b$ 
two positive integers with $a+b\geq 4$ such that  $f(a)=p^{k_1}$, $f(b)=p^{k_2}r$ with $p,r$ prime numbers with $p>r$, and 
$k_1,k_2$ integers with $k_2\geq k_1>0$. Then $f$ is the product of at most $1+\lfloor (k_2-k_1)\log _{r}p \rfloor $ irreducible factors over $\mathbb{Q}$.
\end{corollary} 
\begin{proof} \ Here we use the fact that $q_k=1$ with $k=1+\lfloor (k_2-k_1)\log _{r}p \rfloor $, and also the fact that $M<2$. The conclusion follows again by Theorem \ref{thm1} iii).
\end{proof}
\section{The case of unitary divisors}  \label{se3}
The aim of this section is to find upper bounds for the sum of the multiplicities of the irreducible factors by studying the unitary divisors of $f(a)$ and $f(b)$. Instead of $q_k$ given by 
(\ref{primulq}), we will use here a potentially smaller rational number $q^{u}_{k}$, defined by
\begin{equation}\label{aldoileaq}
q^{u}_{k}=\max \left\{ \frac{d_2}{d_1}\leq \sqrt[k+1]{\frac{|f(b)|}{|f(a)|}}: 
d_1\in \mathcal{D}_u(f(a)), \ d_2\in \mathcal{D}_u(f(b))\right\} .
\end{equation}
With this notation we have the  following results, that are similar to the theorems in the case of admissible divisors.
\begin{theorem}\label{thm0unitary}
Let $f(X)=a_{0}+a_{1}X+\cdots +a_{n}X^{n}\in \mathbb{Z}[X]$, and assume 
that for two integers $a,b$ we have $0<|f(a)|<|f(b)|$ and 
$\gcd(f(a),f'(a))=\gcd(f(b),f'(b))=1$. Let $k$ be a positive intiger and let $q^u_k$ be given by {\em (\ref{aldoileaq})}. Then $f$ is the product of at most $k$ irreducible factors over $\mathbb{Q}$ in each one of the following three cases:

\emph{i)} \ \thinspace $q^u_k>1$ and all the roots of $f$ lie inside the Apollonius  circle ${\rm Ap}(a,b,q^u_k)$;

\emph{ii)} \thinspace $q^u_k>1$, all the roots of $f$ lie inside the Apollonius circle ${\rm Ap}(a,b,\sqrt{q^u_k})$, and $f$ has no rational roots;

\emph{iii)} $q^u_k=1$ and either $b>a$ and all the roots of $f$ lie in 
the half-plane $x<\frac{a+b}{2}$, or $a>b$ and all the roots of $f$ lie in the half-plane $x>\frac{a+b}{2}$.
\end{theorem}

\begin{proof}\ Using the same notations as in the proof of Theorem \ref{thm0}, 
we observe that our assumption that $\gcd(f(a),f'(a))=\gcd(f(b),f'(b))=1$ together 
with the divisibility conditions 
\[
\gcd \left( d_i,\frac{f(a)}{d_i}\right) \mid  \gcd (f(a),f'(a))\ 
\mbox{\rm \ and\ }\ \gcd \left( d'_i,\frac{f(b)}{d'_i}\right) \mid  \gcd (f(b),f'(b))
\]
imply that $d_i$ and $d'_i$ must be unitary divisors of $f(a)$ and $f(b)$, respectively, and this holds for each $i=1,\dots ,s$. The proof continues as in the case of 
Theorem \ref{thm0}, except that instead of $q_{s-1}$ and $q_k$ we have to use here $q^u_{s-1}$ and $q^u_{k}$, respectively, which are still greater than or equal to $1$, as $1$ belongs to both $\mathcal{D}_u(f(a))$ and $\mathcal{D}_u(f(b))$.
\end{proof}
One may rephrase Theorem \ref{thm0unitary} too in a more explicit form, as follows:
\begin{theorem}\label{thm0unitaryExplicit}
Let $f(X)=a_{0}+a_{1}X+\cdots +a_{n}X^{n}\in \mathbb{Z}[X]$, and assume 
that for two integers $a,b$ we have $0<|f(a)|<|f(b)|$ and 
$\gcd(f(a),f'(a))=\gcd(f(b),f'(b))=1$. Let $k$ be a positive integer and let $q^u_k$ be given by {\em (\ref{aldoileaq})}. Then $f$ is the product of at most $k$ irreducible factors over $\mathbb{Q}$ in each one of the following three cases:

\emph{i)} \ \thinspace $q^u_k>1$ and each root $\theta $ of $f$ satisfies $|\theta-\frac{a(q^u_k)^2-b}{(q^u_k)^2-1}|<q^u_k\frac{|b-a|}{(q^u_k)^2-1}$;

\emph{ii)} \thinspace $q^u_k>1$, each root $\theta $ of $f$ satisfies $|\theta -\frac{aq^u_k-b}{q^u_k-1}|<\sqrt{q^u_k}\frac{|b-a|}{q^u_k-1}$, and $f$ has no rational roots;

\emph{iii)} $q^u_k=1$ and either $b>a$ and all the roots of $f$ lie in 
the half-plane $x<\frac{a+b}{2}$, or $a>b$ and all the roots of $f$ lie in the half-plane $x>\frac{a+b}{2}$.
\end{theorem}
\begin{proof} \ The 
proof is similar to that of Theorem \ref{thm0Explicit}, with $q^u_k$ instead of $q_k$.
\end{proof}
We mention here that considerations as in Remark \ref{Remark1} apply in this case too, with $q^u_k$ instead of $q_k$.

\begin{theorem}\label{thm3} 
Let $f(X)=a_{0}+a_{1}X+\cdots +a_{n}X^{n}$ be a polynomial with integer 
coefficients, $M$ the maximum of the absolute values of its roots,
and assume that for two integers $a,b$ we have $0<|f(a)|<|f(b)|$ 
and $\gcd(f(a),f'(a))=\gcd(f(b),f'(b))=1$. Let also $q^u_k$ be given by relation {\em (\ref{aldoileaq})}.
Then $f$ is the product of at most $k$ irreducible factors over $\mathbb{Q}$ in each one of the following three cases:

\emph{i)} \ \thinspace \thinspace $|b|>q^u_k|a|+(1+q^u_k)M$; 

\emph{ii)} \thinspace \thinspace $|b|>\sqrt{q^u_k}|a|+(1+\sqrt{q^u_k})M$ and 
$f$ has no rational roots;

\emph{iii)} \thinspace $q^u_k=1$, $a^2<b^2$ and $M<\frac{|a+b|}{2}$.\end{theorem}
\begin{proof} \ One argues as in the proof of Theorem \ref{thm1}, with $q^u_k$ instead of $q_k$.
\end{proof}
In particular, we obtain the following result.
\begin{corollary} \label{coro3main} 
Let $f(X)=a_{0}+a_{1}X+\cdots +a_{n}X^{n}\in\mathbb{Z}[X]$ and $a$, $b$ two integers such that $a^2<b^2$ and
$|a_{n}|>\sum_{i=0}^{n-1}|a_{i}|(\frac{|a+b|}{2}) ^{i-n}$. Assume that $|f(a)|=p^{k_1}$, $|f(b)|=p^{k_2}r^j$ with $p,r$ distinct prime numbers and $k_{1},k_{2}$ positive integers such $r^j>p^{k_1-k_2}$. Assume also that $p\nmid f'(a)f'(b)$ and $r\nmid f'(b)$. Then $f$ is the product of at most $k$ irreducible factors over $\mathbb{Q}$ in each one of the following cases: 

\emph{i)} $k_1<k_2$, $r^j<p^{k_1}$, $r^j<p^{k_2-k_1}$, and $k=1+\lfloor \frac{k_2-k_1}{j}\log _{r}{p}\rfloor $;

\emph{ii)} $k_1<k_2$, $r^j>p^{k_1}$, $r^j<p^{k_2-k_1}$, and $k=1+\lfloor \frac{k_2\log _rp}{j-k_1\log _rp}\rfloor $;

\emph{iii)} $k_1=k_2$, $r^j>p^{2k_1}$, and $k=\lfloor \frac{j}{k_1}\log _{p}{r}\rfloor $;

\emph{iv)} $k_1=k_2$, $p^{2k_1}>r^j>p^{k_1}$, and $k=1+\lfloor \frac{k_1\log _rp}{j-k_1\log _rp}\rfloor $;

\emph{v)} $k_1=k_2$, $r^j<p^{k_1}$, and $k=1$ ($f$ is irreducible);

\emph{vi)} $k_1>k_2$, $r^j<p^{k_1}$, and $k=1$ ($f$ is irreducible);

\emph{vii)} $k_1>k_2$, $r^j>p^{k_1+k_2}$, and $k=1+\lfloor \frac{j\log _pr-k_1}{k_2}\rfloor $;

\emph{viii)} $k_1>k_2$, $p^{k_1}<r^j<p^{k_1+k_2}$, and $k=1+\lfloor \frac{k_2\log _rp}{j-k_1\log _rp}\rfloor $.

\end{corollary}

\begin{proof} \ As in the proof of 
Corollary \ref{coro1main} we have $M<\frac{|a+b|}{2} $. Note that $\mathcal{D}_u(f(a))=\{ 1,p^{k_1}\} $ and 
$\mathcal{D}_u(f(b))=\{ 1,p^{k_2},r^{j},p^{k_2}r^{j}\} $, so any quotient 
$\frac{d_{2}}{d_{1}}$ with $d_{1}\in \mathcal{D}_u(f(a))$ and $d_{2}\in \mathcal{D}_u(f(b))$ belongs to the set
\[
S=\left\{ 1,\frac{1}{p^{k_1}},p^{k_2},p^{k_2-k_1},r^j,\frac{r^j}{p^{k_1}},p^{k_2}r^j,p^{k_2-k_1}r^j\right\} .
\]
In each case $i$ in the statement denote by $s_i$ the least element of $S$ that exceeds 1. One may check that $s_1=r^j$, $s_2=\frac{r^j}{p^{k_1}}$, $s_3=p^{k_1}$, $s_4=\frac{r^j}{p^{k_1}}$, $s_5=r^j$, $s_6=p^{k_2-k_1}r^j$, $s_7=p^{k_2}$ and $s_8=\frac{r^j}{p^{k_1}}$. The values of $k$ in the statement are then obtained by imposing the conditions $s_i>\sqrt[k+1]{\frac{|f(b)|}{|f(a)|}}$, that is $s_i^{k+1}>p^{k_2-k_1}r^j$, thus forcing $q^u_k$ to be equal to $1$ in each one of the eight cases. 
The conclusion follows from Theorem \ref{thm3}. 
\end{proof}

\section{The case of multivariate polynomials}  \label{se4}
In this section we will prove some similar results for
bivariate polynomials $f(X,Y)$ over an arbitrary field $K$. Generalizations for polynomials $f(X_{1},\dots ,X_{r})$ in $r\geq 3$ variables may be then obtained from the results in the bivariate case, by writing $Y$ for $X_{r}$, $X$ for $X_{r-1}$, and by replacing
$K$ with $K(X_{1},\dots ,X_{r-2})$. We will recall the definition used in \cite{Apollonius} for the admissible divisors in the bivariate case.
\begin{definition}\label{admissible2}
Let $K$ be a field, $f(X,Y)\in K[X,Y]$ and $a(X)\in K[X]$ such that 
$f(X,a(X))\neq 0$. We say that a polynomial $D(X)\in K[X]$ is an 
{\it admissible divisor} of $f(X,a(X))$ if $D(X)\mid f(X,a(X))$ and 
\begin{equation}\label{adm2var}
\gcd \left( D(X),\frac{f(X,a(X))}{D(X)}\right) \mid \gcd 
\left( f(X,a(X)),\frac{\partial f}{\partial Y}(X,a(X))\right) .
\end{equation}
The set of admissible divisors of 
$f(X,a(X))$ will be denoted by $D_{ad}(f(X,a(X))$, and for $f(X,Y)$ and $a(X)$ as above we will denote 
\[
D_{u}(f(X,a(X)))=\left\{ d\in K[X]:d(X)|f(X,a(X)),\ \gcd\left( d(X),\frac{f(X,a(X))}{d(X)}\right) =1\right\} ,
\]
and call it the set of {\it unitary divisors} of $f(X,a(X))$. Notice that 
in the particular case that 
$\gcd (f(X,a(X)),\frac{\partial f}{\partial Y}(X,a(X)))=1$, 
$D_{ad}(f(X,a(X))$ reduces to $\mathcal{D}_u(f(X,a(X)))$.
\end{definition}
With this definition, we have the following result.
\begin{theorem}\label{thm5}
Let $K$ be a field, $f(X,Y)=a_{0}(X)+a_{1}(X)Y+\cdots +a_{n}(X)Y^{n}\in K[X,Y]$, 
with $a_{0},\dots ,a_{n}\in K[X]$, $a_{0}a_{n}\neq 0$, and $k$ a positive integer. Assume that for two polynomials $a(X),b(X)\in K[X]$ we have 
$f(X,a(X))f(X,b(X))\neq 0$ and 
$\Delta _k:=\frac{1}{k+1}\cdot (\deg f(X,b(X))-\deg f(X,a(X)))\geq 0$, and let
\[
q_k=\max \{ \deg d_{2}-\deg d_{1}\leq \Delta _k :d_{1}\in D_{ad}(f(X,a(X))),d_{2}\in D_{ad}(f(X,b(X))) \}.
\]
If $\deg b(X)>\max \{ \deg a(X),\max \limits_{0\leq i\leq n-1 }
\frac{\deg a_{i}-\deg a_{n}}{n-i} \}+q_k$, then $f(X,Y)$ is the product of at most $k$ irreducible factors over $K(X)$.
\end{theorem}

\begin{proof} \ Note that once we fix the polynomials $a(X)$ and $b(X)$ as in the statement of the theorem, $q_k$ is a decreasing function on $k$. Moreover, since $\Delta _k\geq 0$ and $1$ is obviously an admissible divisor 
of both $f(X,a(X))$ and $f(X,b(X))$, a possible candidate for $q_{k}$ is $0=\deg 1-\deg 1$, so $q_{k}\geq 0$ for each $k$.  Assume towards a contradiction that one may write $f$ as a product of $s\geq k+1$ factors $f_1,\dots ,f_s\in K[X,Y]$ that are irreducible over $K(X)$. In particular, we have \begin{equation}\label{qkqs}
q_k\geq q_{s-1}. 
\end{equation}
Let us fix an index $i\in\{ 1,\dots ,s\} $ and denote
\[
g_i(X,Y):=\prod\limits_ {j\neq i}f_j(X,Y).
\]
We may therefore write $f=f_i\cdot g_i$ for each $i=1,\dots ,s$. Next, since 
\begin{eqnarray*}
f(X,a(X)) & = & f_i(X,a(X))g_i(X,a(X))\neq 0\\
f(X,b(X)) & = & f_i(X,b(X))g_i(X,b(X))\neq 0
\end{eqnarray*}
and
\begin{eqnarray*}
\frac{\partial{f}}{\partial{Y}}(X,a(X)) & = & 
\frac{\partial{f_i}}{\partial{Y}}(X,a(X))g_i(X,a(X))+f_i(X,a(X))\frac{\partial{g_i}}{\partial{Y}}(X,a(X)),\\
\frac{\partial{f}}{\partial{Y}}(X,b(X)) & = & 
\frac{\partial{f_i}}{\partial{Y}}(X,b(X))g_i(X,b(X))+f_i(X,b(X))\frac{\partial{g_i}}{\partial{Y}}(X,b(X)),
\end{eqnarray*}
we see that $f_i(X,a(X))$ is a divisor $d_{i}(X)$ of $f(X,a(X))$, and $f_i(X,b(X))$ is a divisor 
$d'_{i}(X)$ of $f(X,b(X))$
that must also satisfy the divisibility conditions
\begin{eqnarray*}
\gcd \left( d_i(X),\frac{f(X,a(X))}{d_i(X)}\right) & \mid & \gcd 
\left( f(X,a(X)),\frac{\partial f}{\partial Y}(X,a(X))\right) \ \mbox{\rm and}\\ 
\gcd \left( d'_i(X),\frac{f(X,b(X))}{d'_i(X)}\right) & \mid & \gcd 
\left( f(X,b(X)),\frac{\partial f}{\partial Y}(X,b(X))\right) . 
\end{eqnarray*}
Therefore $d_i(X)$ and $d'_i(X)$ are admissible divisors of $f(X,a(X))$ and $f(X,b(X))$, 
respectively, and this holds for each $i=1,\dots ,s$. Observe now that
\begin{equation}\label{produsderapoarte}
\frac{d'_1(X)}{d_1(X)}\cdots \frac{d'_s(X)}{d_s(X)}=\frac{f(X,b(X))}{f(X,a(X))}.
\end{equation}
At this point we will introduce as in \cite{Apollonius} a nonarchimedean
absolute value $|\cdot|$ on $K(X)$, as follows. We choose an arbitrary real number
$\rho>1$, and for any polynomial $A(X)\in K[X]$ we define $|A(X)|$ by
\[
|A(X)|=\rho^{\deg A(X)}.
\]
Then we extend the absolute value $|\cdot|$ to $K(X)$ by multiplicativity,
that is, for any polynomials $A(X),B(X)\in K[X]$ with $B(X)\neq 0$, we let   
$\left| \frac{A(X)}{B(X)} \right|=\frac{|A(X)|}{|B(X)|}$.  
We note here that for any non-zero element $A(X)$ of $K[X]$ one has $|A|\geq 1$.
Finally, let $\overline{K(X)}$ be a fixed algebraic closure of $K(X)$, and let us fix an
extension of our absolute value $|\cdot|$ to $\overline{K(X)}$, which we will
also denote by $|\cdot|$.

Applying now our absolute value to relation (\ref{produsderapoarte}), we deduce that
one of the quotients $\frac{|d'_1(X)|}{|d_1(X)|},\dots ,\frac{|d'_s(X)|}{|d_s(X)|}$, 
say $\frac{|d'_1(X)|}{|d_1(X)|}$, must be less than or equal to $\sqrt[s]{\frac{|f(X,b(X))|}{|f(X,a(X))|}}$. In particular, we must have $\deg f_1(X,b(X))-\deg f_1(X,a(X))\leq \frac{1}{s}(\deg f(X,b(X))-\deg f(X,a(X)))$, so
$\deg f_1(X,b(X))-\deg f_1(X,a(X))\leq q_{s-1}$, or equivalently,
\begin{equation}\label{qmultiv}
\frac{|f_1(X,b(X))|}{|f_1(X,a(X))|}\leq \rho ^{q_{s-1}}.
\end{equation}

Let us assume that $f$ as a polynomial in $Y$ with coefficients in $K[X]$ factorizes as 
\[
f(X,Y)=a_{n}(X)(Y-\xi _{1})\cdots (Y-\xi _{n})
\]
for some $\xi _{1},\dots ,\xi _{n}\in \overline{K(X)}$. 
Next, we will prove that 
\begin{equation}\label{grade}
\max \{ |\xi _{1}|,\dots ,|\xi _{n}|  \} \leq \rho ^ {\ \max 
\limits_{0\leq i\leq n-1 }\frac{\deg a_{i}-\deg a_{n}}{n-i}}.
\end{equation}
To this end, let $\delta :=\max \limits_{0\leq i\leq n-1 }
\frac{\deg a_{i}-\deg a_{n}}{n-i}$, and let us assume to the contrary 
that $f$ has a root $\xi $ with $|\xi |> \rho ^{\delta }$. Since 
$\xi \ne 0$ and our absolute value also satisfies the triangle inequality, we successively deduce that
\begin{eqnarray*}
0=\left| \sum\limits _{i=0}^{n}a_{i}\xi ^{i-n}\right| & \geq & 
|a_{n}|- \left| \sum\limits _{i=0}^{n-1}a_{i}\xi ^{i-n}\right|
\geq |a_{n}|- \max\limits _{0\leq i\leq n-1}|a_{i}|\cdot |\xi |^{i-n}\\
& > & |a_{n}|- \max\limits _{0\leq i\leq n-1}|a_{i}|\cdot \rho^{(i-n)\delta },
\end{eqnarray*}
so $|a_{n}|< \max\limits _{0\leq i\leq n-1}|a_{i}|\cdot \rho^{(i-n)\delta }$, 
or equivalently
\begin{equation}\label{gradenou}
\deg a_{n}< \max\limits _{0\leq i\leq n-1}\{ \deg a_{i}+(i-n)\delta \} .
\end{equation}
Let us choose now an index $k\in \{ 0,\dots , n-1\} $ for which the maximum 
in (\ref{gradenou}) is attained. In particular, we have
$\deg a_{n}< \deg a_{k}+(k-n)\delta $,
which leads us to
\[
\frac{\deg a_{k}-\deg a_{n}}{n-k}>\delta =\max \limits_{0\leq i\leq n-1 }\frac{\deg a_{i}-\deg a_{n}}{n-i},
\]
a contradiction. Therefore inequality (\ref{grade}) must hold, so 
$|\xi _{i}|\leq\rho ^{\delta }$ for $i=1,\dots ,n$.

Assume now without loss of generality that $f_1(X,Y)=b_{m}(Y-\xi _1)\cdots (Y-\xi _{m})$ for some $m\geq 1$ and some divisor $b_{m}$ of $a_n$. Notice that we may write
\[
\frac{f_1(X,b(X))}{f_1(X,a(X))}=\frac{b(X)-\xi_1}{a(X)-\xi_1}\cdots 
\frac{b(X)-\xi_{m}}{a(X)-\xi_{m}},
\]
so by (\ref{qmultiv}) we see that for at least one index $j\in \{ 1,\dots ,m\} $ 
we must have
\begin{equation}\label{radicalmultiv}
\frac{|b(X)-\xi_j|}{|a(X)-\xi_j|}\leq \rho ^{\frac{q_{s-1}}{m}}.
\end{equation}
On the other hand, our absolute value also satisfies the triangle 
inequality, so
\[
\frac{|b(X)-\xi_j|}{|a(X)-\xi_j|}\geq \frac{|b(X)|-|\xi_j|}{|a(X)|+|\xi_j|}
\geq \frac{|b(X)|-\rho ^{\delta }}{|a(X)|+\rho ^{\delta }}
=\frac{\rho ^{\deg b(X)}-\rho ^{\delta }}{\rho ^{\deg a(X)}+\rho ^{\delta }}.
\]
All that remains now is to prove that for a sufficiently large $\rho $ we have
\[
\frac{\rho ^{\deg b(X)}-\rho ^{\delta }}{\rho ^{\deg a(X)}+\rho ^{\delta }}>
\rho ^{q_{s-1}}\geq \rho ^{\frac{q_{s-1}}{m}},
\]
which will contradict (\ref{radicalmultiv}). Here the right-most inequality 
$\rho ^{q_{s-1}}\geq \rho ^{\frac{q_{s-1}}{m}}$ obviously holds for an arbitrary
$\rho >1$ since $q_{s-1}\geq 0$ and $m=\deg _{Y}f_1\geq 1$. The first inequality is equivalent to
\[
\rho ^{\deg b(X)}>\rho ^{q_{s-1}+\deg a(X)}+\rho ^{q_{s-1}+\delta }+\rho ^{\delta },
\]
which will obviously hold for a sufficiently large $\rho $, since by our assumption on the magnitude of $\deg b(X)$ we have 
$\deg b(X)>\max \{ q_k+\deg a(X), q_k+\delta ,\delta \} $, and $q_k\geq q_{s-1}$, according to (\ref{qkqs}). Therefore one may write $f$ as a product of at most $k$ irreducible factors over $K(X)$,
and this completes the proof of the theorem.  
\end{proof}

In particular, for $a(X)=0$ and $b(X)$ denoted by $g(X)$, we obtain from Theorem \ref{thm5} 
the following result:
\begin{corollary}\label{coro6} 
Let $K$ be a field, $f(X,Y)=a_{0}(X)+a_{1}(X)Y+\cdots +a_{n}(X)Y^{n}\in K[X,Y]$, 
with $a_{0},a_{1},\dots ,a_{n}\in K[X]$, $a_0a_n\neq 0$, $a_0\in K$ and
\[
\deg a_{n-1}\geq \deg a_{n} \geq\max \{ \deg a_{0},\deg a_{1},\dots ,\deg a_{n-2}\} .
\]
If $f(X,g(X))=h(X)^k$ for $g,h\in K[X]$ with $\deg g>\deg a_{n-1}-\deg a_n$, $h$ irreducible and $k$ a positive integer, then $f(X,Y)$ is the product of at most $k$ irreducible factors over $K(X)$.
\end{corollary}

\begin{proof} \  \ We apply Theorem \ref{thm5} 
with $a(X)=0$ and $b(X)=g(X)$. Since $f(X,a(X))=a_0\in K\setminus \{0\}$ we have $\deg f(X,a(X))=0$, and since $f(X,b(X))=h(X)^k$ with $h$ irreducible over $K$ and $k\geq 1$, we have $\deg f(X,b(X))=k\deg h\geq 1$. Therefore $\Delta _k=\frac{1}{k+1}(\deg f(X,b(X))-\deg f(X,a(X)))=\frac{k}{k+1}\deg h>0$. Our assumption that
\[
\deg a_{n-1}\geq \deg a_{n} \geq\max \{ \deg a_{0},\deg a_{1},\dots ,\deg a_{n-2}\} 
\]
shows that condition $\deg b(X)>\max \{ \deg a(X),\max \limits_{0\leq i\leq n-1 }
\frac{\deg a_{i}-\deg a_{n}}{n-i} \}+q_k$
reduces in this case to $\deg b(X)>\deg a_{n-1}-\deg a_n +q_k$. It remains to prove 
that $q_k=0$. To prove this, we notice that 
any divisor $d_{2}$ of $f(X,b(X))=h(X)^k$ is a power of $h$, as $h$ is irreducible, while any divisor $d_1$ of $f(X,a(X))$ is a constant, as $a_0\in K$. Thus, the least positive value of $\deg d_2-\deg d_1$ with $d_2\mid f(X,b(X))$ and $d_1\mid f(X,a(X))$ is $\deg h$, and since $\Delta _k=\frac{k}{k+1}\deg h<\deg h$, we conclude that $q_k=0$, which completes the proof.  
\end{proof}

We end this section with a result that requires knowing only the unitary divisors of $f(X,a(X))$ and $f(X,b(X))$, provided $f(X,a(X))$ and 
$\frac{\partial f}{\partial Y}(X,a(X))$ are relatively prime, and 
$f(X,b(X))$ and $\frac{\partial f}{\partial Y}(X,b(X))$ are also relatively prime. 
\begin{theorem}\label{thm7}
Let $K$ be a field, $f(X,Y)=a_{0}(X)+a_{1}(X)Y+\cdots +a_{n}(X)Y^{n}\in K[X,Y]$, 
with $a_{0},\dots ,a_{n}\in K[X]$, $a_{0}a_{n}\neq 0$, and $k$ a positive integer. Assume that for two polynomials $a(X),b(X)\in K[X]$ we have 
$f(X,a(X))f(X,b(X))\neq 0$ and 
$\Delta _k:=\frac{1}{k+1}\cdot (\deg f(X,b(X))-\deg f(X,a(X)))\geq 0$, and let
\[
q^{u}_k=\max \{ \deg d_{2}-\deg d_{1}\leq \Delta _k :d_{1}\in D_{u}(f(X,a(X))),d_{2}\in D_{u}(f(X,b(X))) \}.
\]
If $\gcd(f(X,a(X)), \frac{\partial f}{\partial Y}(X,a(X)))=1$,
$\gcd(f(X,b(X)), \frac{\partial f}{\partial Y}(X,b(X)))=1$ and 
\[
\deg b(X)>\max \left\{ \deg a(X),\max \limits_{0\leq i\leq n-1 }
\frac{\deg a_{i}-\deg a_{n}}{n-i} \right\}+q^{u}_k,
\]
then $f(X,Y)$ is the product of at most $k$ irreducible factors over $K(X)$.
\end{theorem}
\begin{proof} \ Here, with the same notations as in the proof of 
Theorem \ref{thm5}, we see that $d_{i}(X)$ must belong to 
$\mathcal{D}_u(f(X,a(X)))$, while $d_i'(X)$ must belong 
to $\mathcal{D}_u(f(X,b(X)))$. We notice here that we will still have $q^{u}_{k}\geq 0$, since $1$ belongs to both 
$\mathcal{D}_u(f(X,a(X)))$ and $\mathcal{D}_u(f(X,b(X)))$. The rest of the proof 
is similar to that of Theorem \ref{thm5}, and will be omitted.
\end{proof}

\section{Examples} \label{se5}

{\bf 1)\ } To show that in some cases our results are sharp, we will first consider the following example. Let $p\geq 7$ be a prime number, and let $f(X)=p(p-1)X^3+X^2+(p-2)X+1$. We will apply Corollary \ref{coro2} by observing that $f(1)=p^2$, $a_0=1$, so $k=2$ and $1+\lfloor (k-1)\log _{\frac{p}{|a_0|}}p \rfloor =2$. The condition $|a_{n}|>2|a_{n-1}|+2^{2}|a_{n-2}|+\cdots +2^{n}|a_{0}|$ reduces in our case to the inequality $p(p-1)>2+4(p-2)+8$, that is to $p^2>5p+2$, which holds for primes $p\geq 7$. We may thus conclude that $f$ is the product of at most two irreducible factors over $\mathbb{Q}$. On the other hand, we notice that $f$ may be written as
$(pX^2-X+1)((p-1)X+1)$, so it has two irreducible factors.  

{\bf 2)\ } For another simple example where our results provide sharp estimates, this time with $q_2>1$, consider the polynomial $f(X)=35X^4+12X^2+1$, and let us pretend that we don't know how to factor it. One may check that $f$ has no rational roots. Without computing the roots of $f$, we are going to test the conditions in Theorem \ref{thm0Explicit} ii) with $a=1$ and $b=2$. We have $f(1)=2^4\cdot 3$ and $f(2)=3\cdot 7\cdot 29$, and one may deduce that $q_2=\frac{29}{16}=1.8125$. Instead of asking all the roots $\theta$ of $f$ to satisfy condition $|\theta -\frac{aq_2-b}{q_2-1}|<\sqrt{q_2}\frac{|b-a|}{q_2-1}$, it suffices to check that they satisfy $|\theta |<\sqrt{q_2}\frac{|b-a|}{q_2-1}-\frac{|aq_2-b|}{q_2-1}\approx 1.4262$. Since $35>12+1$, by Rouch\'e' s Theorem all the roots of $f$ have absolute values less than $1$, and we conclude by Theorem \ref{thm0Explicit} ii) that $f$ is the product of at most two irreducible factors over $\mathbb{Q}$.

{\bf 3)\ } Let us fix any arbitrarily chosen integers 
$a_{1},\dots ,a_{n-1}$ and $k\geq 0$. Then for all but finitely many prime numbers $p$ the polynomial
\[
f(X)=2p^{k}+a_{1}X+\cdots +a_{n-1}X^{n-1}+(p^{k+2}-2p^{k}-a_{1}-\dots -a_{n-1})X^{n}
\]
is the product of at most two irreducible factors over $\mathbb{Q}$. 

To prove this, observe that $f(0)=2p^{k}$ and $f(1)=p^{k+2}$, so we may apply 
Corollary \ref{coro1main} with $a=0$, $b=1$, $k_1=k$, $k_2=k+2$, $r=2$, and $1+\lfloor (k_2-k_1-1)\log _{\frac{p}{r}}p \rfloor =2$, since $1<\log _{\frac{p}{r}}p<2$ for $p\geq 5$. It remains to prove that the leading coefficient of $f$ satisfies the inequality
\[
|p^{k+2}-2p^{k}-a_{1}-\dots -a_{n-1}|>2^{n+1}p^{k}+\sum\limits_{i=1}^{n-1}2^{n-i}|a_{i}|,
\]
and this obviously holds for sufficiently large prime numbers $p$.

For an example where an explicit lower bound for $p$ can be easily obtained,
one can take $|a_{1}|=\dots =|a_{n-1}|=1$ and $k\geq 2$ to conclude that $f$  
is the product of at most two irreducible factors over $\mathbb{Q}$ for all primes $p\geq 2^{\frac{n+1}{2}}+1$.
Indeed, since $|a_n|\geq p^{k+2}-2p^{k}-n+1$ and $\sum_{i=0}^{n-1}2^{n-i}|a_i|=2^{n}-2+2^{n+1}p^{k}$, it suffices to ask $p$ to satisfy $p^2>2+2^{n+1}+\frac{2^n+n-3}{p^k}$,
which will obviously hold for $p\geq 2^{\frac{n+1}{2}}+1$.

\smallskip 

{\bf 4)\ } Let $p$ be a prime number and consider the polynomial $f(X,Y)\in \mathbb{Z}[X,Y]$ given by
\[
(pX+1)^2Y^4+(2p-2)(pX^2+X)Y^3+[(p-1)^2X^2+2pX+2]Y^2+(2p-2)XY+1.
\]
Notice that $f$ is written as a polynomial in $Y$ with coefficients 
in $\mathbb{Z}[X]$, namely $f=\sum _{i=0}^{4}a_{i}(X)Y^{i}$ with 
$a_{i}(X)\in \mathbb{Z}[X]$, and we have $\deg a_{4}=\deg a_{3}=\deg a_{2}=2$, $\deg a_{1}=1$, and $\deg a_{0}=0$, so condition 
\[
\deg a_{n-1}\geq \deg a_{n} \geq\max \{ \deg a_{0},\deg a_{1},\dots ,\deg a_{n-2}\} 
\]
in Corollary \ref{coro6} is satisfied. On the other hand, we observe that
\begin{eqnarray*}
	f(X,X) & = & p^2X^{6}+2p^2X^5+p^2X^4+2pX^3+2pX^2+1\\
	& = & (pX^3+pX^2+1)^2, 
\end{eqnarray*}
which is the square of an Eisensteinian polynomial with respect to the prime $p$. We may thus apply Corollary \ref{coro6} with $g(X)=X$, which satisfies the condition $\deg g>\deg a_{n-1}-\deg a_n$, and with $h(X)=pX^3+pX^2+1$, which is irreducible over $\mathbb{Q}$. We conclude that $f$ is the product of at most two irreducible factors over $\mathbb{Q}(X)$. Indeed, one may check that in fact we have $f(X,Y)=((pX+1)Y^2+(p-1)XY+1)^2$, so $(pX+1)Y^2+(p-1)XY+1$ must be irreducible over $\mathbb{Q}(X)$, which may also be tested directly, or again by Corollary \ref{coro6}.

\medskip

{\bf Acknowledgements} \ This work was done in the frame of the GDRI ECO-Math.

\end{document}